\documentclass[11pt]{article}

\usepackage[margin=1in]{geometry}
\usepackage{amsmath,amssymb,amsthm,mathtools}
\usepackage{booktabs,array,tabularx}
\usepackage[dvipsnames]{xcolor}
\usepackage{tikz}
\usetikzlibrary{arrows.meta,calc,positioning,fit,decorations.pathreplacing,shapes.geometric,backgrounds}
\usepackage{enumitem}
\usepackage{caption}
\usepackage{hyperref}
\hypersetup{colorlinks=true,linkcolor=blue,citecolor=cyan,urlcolor=blue,filecolor=blue}

\definecolor{RedGraph}{RGB}{205,30,45}
\definecolor{BlueGraph}{RGB}{0,105,200}
\definecolor{SoftRed}{RGB}{255,235,235}
\definecolor{SoftBlue}{RGB}{232,243,255}
\definecolor{SoftGray}{RGB}{245,245,245}
\definecolor{DarkGray}{RGB}{85,85,85}

\newtheorem{theorem}{Theorem}

\newtheorem{claim}{Claim}
\newtheorem{conjecture}{Conjecture}

\newtheorem{fact}{Fact}
\theoremstyle{definition}

\newcommand{\E}{\mathbb E}
\newcommand{\Pbb}{\mathbb P}
\newcommand{\red}{\mathrm{red}}
\newcommand{\blue}{\mathrm{blue}}
\newcommand{\fall}[2]{(#1)_{#2}}

\title{On the threshold Ramsey multiplicity conjectures for paths and even cycles}
\author{Ting HUANG, Jiabao YANG, Yaojun CHEN\footnote{Corresponding author. Email: yaojunc@nju.edu.cn}\\
 \small{School of Mathematics, Nanjing University, Nanjing 210093, P.R. CHINA}}
\date{}

\begin{document}
\maketitle

\begin{abstract}
The Ramsey number $r(H)$ of a graph $H$ is the minimum positive integer $n$ such that every red/blue edge-coloring of the complete graph $K_n$ on $n$ vertices contains a monochromatic copy of $H$.  The threshold Ramsey multiplicity $m(H)$ of $H$ is the minimum number of monochromatic copies of $H$ over all red/blue edge-colorings of $K_{r(H)}$.
Let $P_t$ and $C_t$ be a path and a cycle on $t$ vertices, respectively.
In this paper, by using combinatorial and local random construction, we show that 
$$m(C_{2t})\le t^{-\gamma+o(1)}\frac{(2t-1)!}{2}, \qquad
m(P_{2t+1})\le t^{-\gamma+o(1)}\frac{t}{2}(2t)!,$$
and
$$m(P_{2t})\leq \left(\frac{7}{8}+o(1)\right)\frac{(2t)!}{2},$$
for sufficiently large $t$, where $\gamma=1/(1+\sqrt{2})$. These results disprove two conjectures on the threshold Ramsey multiplicity for even cycles and paths, due to Conlon, Fox, Sudakov, and Wei.

\end{abstract}


\section{Introduction}
The Ramsey number $r(H)$ of a graph $H$ is the minimum positive integer $n$ such that every red/blue coloring of the edges of the complete graph $K_n$ on $n$ vertices contains a monochromatic copy of $H$.
Determining classical Ramsey numbers is a very challenging problem, and the exact value of $r(H)$ is known only in a few special cases. For example, the Ramsey number of $K_5$ is still unknown.
In fact, even for generalized Ramsey numbers, there are few families of graphs, including paths and cycles, whose Ramsey numbers are determined. 
Let $P_t$ and $C_t$ denote a path and a cycle on $t$ vertices, respectively.
Gerencs\'er and Gy\'arf\'as \cite{Gerencser1967} determined the Ramsey numbers of paths.

\begin{theorem}[Gerencs\'er and Gy\'arf\'as \cite{Gerencser1967}]\label{thm:Ramsey-paths}
For every integer $t$, 
$$r(P_t)=t-1+\left\lfloor\frac{t}{2}\right\rfloor.$$
\end{theorem}

For cycles, it is trivial that $r(C_3)=r(C_4)=6$. 
The general case for cycles was proved by Faudree and Schelp \cite{Faudree1974}, and independently by Rosta \cite{Rosta1973}.

\begin{theorem}[Faudree and Schelp \cite{Faudree1974}, Rosta \cite{Rosta1973}]\label{thm:Ramsey-cycles}
For every integer $t\geq 5$,
\begin{eqnarray*}
 r(C_t)=
 \begin{cases}
 \dfrac{3t}{2}-1, & t\text{ is even},\\
 2t-1, & t\text{ is odd}.
 \end{cases}
\end{eqnarray*}  
\end{theorem}

Apart from trying to determine the Ramsey number of a graph, another interesting problem  is to study the Ramsey multiplicity. 
That is, for a graph $H$ and an integer $n$, let $M(H,n)$ be the minimum number of monochromatic copies of $H$ in any red/blue edge-coloring of $K_n$. 
Equivalently, if a coloring $\chi$ has $N_{\red}(H,\chi)$ red copies and $N_{\blue}(H,\chi)$ blue copies of $H$, then
$$M(H,n)=\min_{\chi}\bigl(N_{\red}(H,\chi)+N_{\blue}(H,\chi)\bigr).$$

In 1962, Erd\H{o}s~\cite{Erdos1962} conjectured that, for every complete graph $H$, 
the quantity $M(H,n)$ is asymptotically equal to the expected number of monochromatic copies of $H$ in a uniformly random two-coloring of the edges of $K_n$. 
This conjecture was later generalized by Burr and Rosta~\cite{BurrRosta1980} to all graphs $H$. 
It is true for $K_3$, by a theorem of Goodman~\cite{Goodman1959} that predates Erd\H{o}s' conjecture and almost certainly served as one of its motivations. 
Nevertheless, Thomason~\cite{Thomason1989} proved that the conjecture is already false for $K_4$. 
The failure of this natural conjecture has not diminished interest in the problem: the asymptotic behavior of $M(H,n)$, has been studied extensively; see, for instance, 
\cite{Conlon2012,Fox2008,GrzesikLeeLidickyVolec2020,JaggerStovicekThomason1996,Sidorenko1993,Sidorenko1994,Simonovits1984}.

By the definition of the Ramsey number, $M(H,n)=0$ if and only if $n<r(H)$.
It is therefore natural to ask for the value of $M(H,n)$ at the threshold
where it first becomes positive, namely at $n=r(H)$. We denote this value by
$$m(H)=M(H,r(H))$$
and refer to it as the threshold Ramsey multiplicity of $H$.

The threshold Ramsey multiplicity was first studied systematically by Harary and Prins~\cite{HararyPrins1974}.
In general, determining this parameter, or even obtaining a nontrivial lower
bound for it, appears to be a very difficult problem.
The only family for which $m(H)$ has been completely determined is stars, as shown by Harary and Prins~\cite{HararyPrins1974}.
In the same paper, Harary and Prins asked for the threshold Ramsey multiplicity of paths and cycles.
Rosta and Sur\'anyi~\cite{RostaSuranyi1976} studied the case of odd cycles, 
who obtained an exponential lower bound $m(C_t)\ge 2^{ct}.$
K\'arolyi and Rosta~\cite{Karolyi} improved the estimation to $m(C_t)\ge t^{ct}$,
which, as will be seen below, is sharp up to the value of the constant in the exponent.
More recently, Conlon, Fox, Sudakov, and Wei \cite{Conlon2022oddcycle,Conlon2023paths} proved that there exists an absolute constant $c>0$ such that $m(P_t)\ge (ct)^t,$ and an analogous lower bound holds for cycles.
They also conjectured exact formulas for sufficiently large $t$ based on some red/blue edge-colorings of $K_{r(P_t)}$ and $K_{r(C_t)}$.

\begin{conjecture}[Conlon, Fox, Sudakov, and Wei \cite{Conlon2023paths}]\label{conj:CFSW-even}
For sufficiently large $t$,
$$m(C_{2t})=\frac{2t-3}{2}(2t-2)!.$$
\end{conjecture}

\begin{conjecture}[Conlon, Fox, Sudakov, and Wei \cite{Conlon2023paths}]\label{conj:conjecture-path}
For sufficiently large $t$,
\begin{eqnarray*}
 m(P_{2t+1})=\frac{t}{2} (2t)!~\text{ and }\, ~m(P_{2t})=\frac{(2t)!}{2}.
\end{eqnarray*}
\end{conjecture}

The conjectured values above all come from the same critical split coloring.
Let $\chi(a,b)$ be the coloring of $K_{a+b}$ obtained by partitioning the
vertex set into two parts of sizes $a$ and $b$, coloring both parts blue, and
coloring all cross-edges red. This coloring is extremal for the Ramsey numbers
of paths and cycles, and Conlon, Fox, Sudakov, and Wei conjectured that the
corresponding threshold colorings are also extremal for multiplicity.

For $C_{2t}$, the conjectured coloring is obtained from
$\chi(2t,t-1)$ by changing one edge inside the blue clique of order $2t$ to
red. This creates no red $C_{2t}$, and the remaining monochromatic cycles are
the blue Hamilton cycles in $K_{2t}$ avoiding that edge, giving
\[
        \frac{(2t-1)!}{2}-(2t-2)!
        =
        \frac{2t-3}{2}(2t-2)!.
\]

For $P_{2t+1}$, the coloring $\chi(2t,t)$ contains no blue $P_{2t+1}$, and
every red copy must alternate between the two parts. Hence the number of red
copies is
\[
        \frac{t!}{2}(2t)_{t+1}
        =
        \frac{t}{2}(2t)!,
\]
where $(m)_n=m(m-1)\cdots (m-n+1)$ for integers $m\ge n\ge 1$.

For $P_{2t}$, the coloring $\chi(2t,t-1)$ contains no red $P_{2t}$, so all
monochromatic copies come from the blue clique of order $2t$, giving
\(
        (2t)!/2.
\)

The weakness of these split colorings is that they leave too many highly symmetric monochromatic structures. For even cycles, almost all Hamilton cycles of a $2t$-vertex clique remain blue. For odd paths, the complete bipartite red structure creates many alternating paths. 
For even paths, the blue clique of size $2t$ creates many paths.
 
Our constructions follow the same Ramsey‑critical split idea, but we modify the part that creates too many monochromatic copies. We keep the global split structure, but add a local restriction that
forces any remaining copy to pass through a controlled set of edges.

For even cycles, the blue clique has size $2t-1$, so it is one vertex short of
containing a blue $C_{2t}$. Any blue cycle must enter the red clique of size $t$ through
sparse random blue portal edges, while red cycles are forced to alternate
through the red clique and are reduced by the corresponding missing red edges.

For odd paths, the blue clique has size $2t$, so every blue $P_{2t+1}$ must
use vertices outside the clique. These outside vertices, which forms a red clique of size $t$, can be inserted only
through sparse blue support sets. In the red graph, the same support sets
become forbidden adjacencies, deleting many alternating red paths.

For even paths, a random sparse interface is less effective because the
endpoints of a path create additional configurations.  Instead, we use a
deterministic local defect.  Starting from the critical split coloring, we do
not add the last vertex as a full member of the large blue clique.  Rather,
we delete a small blue star from that clique.  This destroys many blue
Hamilton paths, while the new red paths created by the defect all pass
through one common center and remain controllable.

Our main results are the following.

\begin{theorem}\label{thm:main-cycle}
For sufficiently large $t$,
$$m(C_{2t})\le t^{-1/(1+\sqrt{2})+o(1)}\frac{(2t-1)!}{2}.$$
\end{theorem}

\begin{theorem}\label{thm:main-odd-path}
For sufficiently large $t$,
$$m(P_{2t+1})\le t^{-1/(1+\sqrt{2})+o(1)}\frac{t}{2}(2t)!.$$
\end{theorem}

\begin{theorem}\label{thm:main-theorem-even-path}
For every integer $t\geq 3$, we have 
$$m(P_{2t}) \leq\frac{(2t)!}{2}-\left\lfloor \frac{t^2}{4}\right\rfloor(2t-2)!.$$
In particular, as $t\to\infty$, we obtain
$$m(P_{2t})\leq \left(\frac{7}{8}+o(1)\right)\frac{(2t)!}{2}.$$
\end{theorem}

By the three theorems above, Conjectures~\ref{conj:CFSW-even} and~\ref{conj:conjecture-path} are false. Moreover, Theorems~\ref{thm:main-cycle} and \ref{thm:main-odd-path} show the true values of $m(C_{2t})$ and $m(P_{2t+1})$ are much less than the conjectured ones. Although the exact values are still unknown, the combinatorial and local random construction method applied in this paper may shed some light on the determination of threshold Ramsey multiplicity of paths and cycles.

The detailed proofs of Theorems~\ref{thm:main-cycle}, \ref{thm:main-odd-path},  and~\ref{thm:main-theorem-even-path} will be presented in Sections \ref{sec:count-even-cycles}, \ref{sec:count-odd-paths}, and \ref{sec:even paths}, respectively, and we divide each of the proofs into the following three parts: a red/blue edge-coloring construction, a count of monochromatic cycles/paths in the given coloring, and the asymptotic ratio between the resulting upper bound and the conjectured value of Conlon, Fox, Sudakov, and Wei.

\section{Proof of Theorem \ref{thm:main-cycle}}\label{sec:count-even-cycles}
\subsection{The random-portal coloring for even cycles}\label{sec:even-odd-construction}
By Theorem \ref{thm:Ramsey-cycles}, $r(C_{2t})=3t-1$.
We introduce a local random construction for a red/blue edge-coloring of the complete graph $K_{3t-1}$.
Fix integers $t\ge 1$ and $q$ with $1\le q\le t-2$. The value of $q$ will be chosen later. 
Let
$$V(K_{3t-1})=A\cup B,\,\, |A|=t,\,\,~\text{and}~|B|=2t-1,$$
where $A$ and $B$ are pairwise disjoint. 
For each vertex $b\in B$, choose a $q$-element subset $M_b\subseteq A.$
We call $M_b$ the set of \emph{blue portals} from $b$ into $A$.
The edge-coloring of $K_{3t}$ is defined as follows:
\begin{itemize}[leftmargin=2em]
\setlength{\itemsep}{1pt}
\setlength{\parsep}{1pt}
\setlength{\parskip}{1pt}
\item all edges inside $A$ are red;
\item all edges inside $B$ are blue;
\item for $a\in A$ and $b\in B$, the edge $ab$ is blue if and only if $a\in M_b$; otherwise $ab$ is red.
\end{itemize}
The described red/blue coloring of $K_n$ is illustrated in Figure~\ref{fig:AB-handstyle}.

\begin{figure}[htbp]
\centering
\begin{tikzpicture}[scale=1.05, every node/.style={font=\small}]
\tikzset{
  boxred/.style={draw=red, thick, rounded corners=3pt, fill=SoftRed},
  boxblue/.style={draw=blue, thick, rounded corners=3pt, fill=SoftBlue},
  boxgray/.style={draw=DarkGray, thick, rounded corners=2pt, fill=white},
  rededge/.style={draw=red, very thick},
  blueedge/.style={draw=blue, very thick},
  info/.style={font=\small, text=DarkGray}
}

\draw[boxred] (1.2,4.3) rectangle (8.8,6.1);
\node[font=\large] at (5.0,5.85) {$A$:~$|A|=t$};

\draw[draw=DarkGray, thick, fill=white] (3.4,5.2) ellipse (0.97 and 0.43);
\node at (3.4,5.2) {$M_b$};

\draw[boxblue] (1.0,0) rectangle (9.0,2.35);
\node[font=\large] at (5.0,0.35) {$B$:~$|B|=2t-1$};

\draw[blueedge] (4.8,1.15) -- (3.0,5.2);
\draw[blueedge] (4.8,1.15) -- (3.8,5.2);

\draw[rededge] (4.8,1.15) -- (6.8,5.2);
\draw[rededge] (4.8,1.15) -- (7.6,5.2);

\filldraw[draw=black, fill=black, thick] (3.0,5.2) circle (0.05);
\filldraw[draw=black, fill=black, thick] (3.8,5.2) circle (0.05);

\filldraw[draw=black, fill=black, thick] (6.8,5.2) circle (0.06);
\filldraw[draw=black, fill=black, thick] (7.6,5.2) circle (0.06);

\filldraw[draw=black, fill=black, thick] (4.8,1.15) circle (0.06);
\node[below] at (4.8,1.15) {$b$};

\node[info, text=blue, align=center, fill=white, inner sep=1pt] at (2.35,3.45)
  {$b$ is blue-adjacent to $M_b$};

\node[info, text=red, align=center, fill=white, inner sep=1pt] at (8.05,3.45)
  {$b$ is red-adjacent to $A\setminus M_b$};

\end{tikzpicture}
\caption{The random-portal construction for even cycles.}
\label{fig:AB-handstyle}
\end{figure}
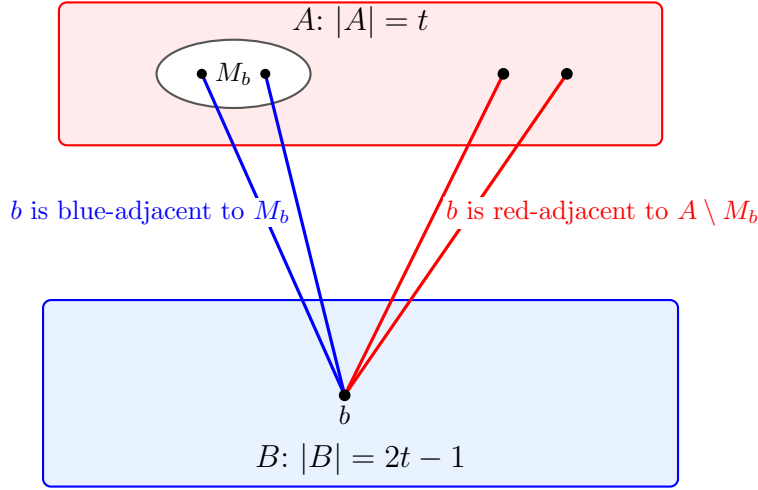

The construction will be used probabilistically. For each $b\in B$, choose $M_b$ independently and uniformly from all $q$-subsets of $A$. We determine the expected number of monochromatic $C_{2t}$'s. If this expectation is at most $\mathbb{E}$, then at least one deterministic choice of the sets $M_b$ gives at most $\mathbb{E}$ monochromatic even cycles.

\subsection{Counting monochromatic even cycles}
For convenience, we define $$T_t^{C}:=\frac{(2t-1)!}{2}.$$
This is the number of Hamilton cycles in the complete graph $K_{2t}$.  Indeed, there are $(2t)!$ linear orderings of the vertices, and each undirected Hamilton cycle is counted $2t$ times by choosing a starting point and twice more by choosing a direction.  Hence the number is $$\frac{(2t)!}{2\cdot 2t}=\frac{(2t-1)!}{2}.$$

\subsubsection{The expected number of red cycles}

\begin{claim}\label{claim:red-cycle-exact}
The expected number of red copies of $C_{2t}$ is
$$\E\bigl(N_{\red}(C_{2t})\bigr)=T_t^{C}\rho_q^t,$$
where
$$\rho_q=\frac{\binom{t-2}{q}}{\binom{t}{q}}=\frac{(t-q)(t-q-1)}{t(t-1)}.$$
\end{claim}

\begin{proof}
In the red graph, the set $B$ is an independent set. 
It follows that no two consecutive vertices on a red cycle can both belong to $B$.
Hence every red $C_{2t}$ uses all vertices of $A$, uses exactly $t$ vertices of $B$, and alternates between $A$ and those chosen vertices of $B$.

The number of choices for the $t$ vertices of $B$ is $\binom{2t-1}{t}$. For a fixed $t$-set $Y\subseteq B$, the number of alternating Hamilton cycles in the complete bipartite graph between $A$ and $Y$ is $t!(t-1)!/2$.
Thus the total number of candidate alternating cycles is
$$\binom{2t-1}{t}\frac{t!(t-1)!}{2}=\frac{(2t-1)!}{2}=T_t^{C}.$$

Now we fix one candidate alternating cycle. Every chosen vertex $b\in B$ has exactly two neighbours in $A$ on the cycle. The two corresponding edges are red if and only if neither of these two $A$-vertices lies in $M_b$. 
Since $M_b$ is a uniformly random $q$-subset of $A$, this probability is $$\rho_q=\frac{\binom{t-2}{q}}{\binom{t}{q}}=\frac{(t-q)(t-q-1)}{t(t-1)}.$$
The choices of the sets $M_b$ are independent for different vertices $b$.  Therefore the probability that the fixed candidate alternating cycle is red is $\rho_q^t$. 
By linearity of expectation, the expected number of red copies of $C_{2t}$ is
$$\E\bigl(N_{\red}(C_{2t})\bigr)=T_t^{C}\rho_q^t,$$
as required.
\end{proof}

\begin{claim}\label{claim:red-cycle-bound}
If $q=o(\sqrt t)$, then
$$\E\bigl(N_{\red}(C_{2t})\bigr)\le T_t^{C}\exp\left(-2q+O\left(\frac{q^2}{t}\right)\right).$$
\end{claim}

\begin{proof}
By Claim~\ref{claim:red-cycle-exact}, it is enough to estimate $\rho_q^t$. 
Note that
$$\rho_q=\frac{(t-q)(t-q-1)}{t(t-1)}=1-\frac{2q}{t}+O\left(\frac{q^2}{t^2}\right).$$
Since $q=o(\sqrt{t})$, we have $q/t=o(1)$. Applying the Taylor expansion of $\log(1+x)$ around $x=0$, we obtain
$$\log \rho_q=-\frac{2q}{t}+O\left(\frac{q^2}{t^2}\right).$$
Multiplying by $t$ yields
$$t\log \rho_q=-2q+O\left(\frac{q^2}{t}\right),$$ and hence
$$\rho_q^t=\exp\left(-2q+O\left(\frac{q^2}{t}\right)\right),$$
which proves Claim~\ref{claim:red-path-bound}.
\end{proof}

\subsubsection{The expected number of blue cycles}

\begin{claim}\label{claim:blue-cycle-sum}
The expected number of blue copies of $C_{2t}$ satisfies
$$\E\bigl(N_{\blue}(C_{2t})\bigr)\le
T_t^{C}\sum_{s=1}^{t}\binom{t}{s}\binom{2t-1}{s-1}\left(\frac{q}{t}\right)^{2s}.$$
\end{claim}

\begin{proof}
The blue graph has a large clique on $B$, and $|B|=2t-1$.  Hence every blue $C_{2t}$ must use at least one vertex of $A$.

Suppose that a blue $C_{2t}$ uses exactly $s$ vertices of $A$, where $1\le s\le t$.  Since $A$ is blue-independent, no two vertices of $A$ can be consecutive on the blue cycle.  The remaining $2t-s$ vertices of the cycle are chosen from $B$.  Therefore the number of possible vertex sets is
$$\binom{t}{s}\binom{2t-1}{2t-s}=\binom{t}{s}\binom{2t-1}{s-1}.$$

For any fixed set of $2t$ vertices, the number of possible undirected cyclic orders is at most $T_t^{C}=(2t-1)!/2$. This is an overcount, because many cyclic orders put two vertices of $A$ next to each other, but an overcount is enough for an upper bound.

Fix one cyclic order. If it is to become a blue cycle, every edge incident with a vertex of $A$ in the cycle must be a blue portal edge. There are exactly $2s$ such incidences.  If a vertex $b\in B$ is required to connect to $r_b$ prescribed vertices of $A$, where $r_b=0,1,2$, then
$$\Pbb(\text{all these }r_b\text{ vertices lie in }M_b)
=\frac{\binom{t-r_b}{q-r_b}}{\binom{t}{q}}
=\frac{\fall{q}{r_b}}{\fall{t}{r_b}}\le \left(\frac{q}{t}\right)^{r_b}.$$
Here $\fall n r=n(n-1)\cdots(n-r+1)$ is the falling factorial, and the inequality holds because, given $q\leq t$, for any $i\geq 0$, we have
$(q-i)/(t-i)\leq q/t$.
It implies that $\fall{q}{r_b}/\fall{t}{r_b}\le \left(q/t\right)^{r_b}$.

The sets $M_b$ are independent for different vertices $b$, and the total value of $\sum_b r_b$ is $2s$. Then the probability that all required portal edges are present is at most 
$$\prod_{b } \left( \frac{q}{t} \right)^{r_b} = \left( \frac{q}{t} \right)^{\sum_{b} r_b} = \left( \frac{q}{t} \right)^{2s}.$$

Multiplying the number of choices in each case and then summing over $s$, we obtain the following upper bound for the expected number of blue copies of $C_{2t}$:
$$\E\bigl(N_{\blue}(C_{2t})\bigr)\le
T_t^{C}\sum_{s=1}^{t}\binom{t}{s}\binom{2t-1}{s-1}\left(\frac{q}{t}\right)^{2s}.$$
The proof of the claim is complete.
\end{proof}

We now estimate the sum in Claim~\ref{claim:blue-cycle-sum}.  We shall use only two standard facts about the modified Bessel function.
The modified Bessel function of the first kind of order one is defined by the
absolutely convergent power series
\[
        I_1(x)
        =
        \sum_{m=0}^{\infty}
        \frac{1}{m!(m+1)!}
        \left(\frac{x}{2}\right)^{2m+1}.
\]
See, for example, the NIST Handbook of Mathematical Functions~\cite[Chapter~10]{NISTHandbook}.
We shall also use the standard large-argument asymptotic formula
\begin{equation}\label{eq:bessel-asymp} I_1(x)=\frac{e^x}{\sqrt{2\pi x}}\left(1+O\left(\frac1x\right)\right),\qquad x\to\infty. \end{equation}
for positive real $x$.  This is the case $\nu=1$ of the usual asymptotic expansion
for the modified Bessel function $I_\nu(x)$; see again~\cite[Chapter~10]{NISTHandbook}.

\begin{claim}\label{claim:blue-cycle-bound}
If $q\to\infty$, then
$$\E\bigl(N_{\blue}(C_{2t})\bigr)\le T_t^{C}\,t^{-1}\exp(2\sqrt2 q+o(q)).$$
\end{claim}

\begin{proof}
For any $1\le s\le t$, we have
$$\binom{t}{s}\binom{2t-1}{s-1}\left(\frac{q}{t}\right)^{2s} \le \frac{t^s}{s!}\cdot\frac{(2t)^{s-1}}{(s-1)!}\cdot\frac{q^{2s}}{t^{2s}} \le \frac{2^s q^{2s}}{t\,s!(s-1)!}.$$
Thus Claim~\ref{claim:blue-cycle-sum} gives
$$\E\bigl(N_{\blue}(C_{2t})\bigr)
        \le
        \frac{T_t^{C}}{t}
        \sum_{s=1}^{t}
        \frac{2^s q^{2s}}{s!(s-1)!}
        \le
        \frac{T_t^{C}}{t}
        \sum_{s=1}^{\infty}
        \frac{2^s q^{2s}}{s!(s-1)!}.$$

Putting $s=m+1$, we get
$$\sum_{s=1}^{\infty}\frac{2^s q^{2s}}{s!(s-1)!}=\sum_{m=0}^{\infty}\frac{2^{m+1}q^{2m+2}}{(m+1)!m!}=
        \sqrt2 q
        \sum_{m=0}^{\infty}
        \frac{(\sqrt2 q)^{2m+1}}{m!(m+1)!}= \sqrt2 q I_1(2\sqrt2 q).$$
Using \eqref{eq:bessel-asymp} with $x=2\sqrt2 q$, we get, whenever $q\to\infty$,
$$\E\bigl(N_{\blue}(C_{2t})\bigr)
\le \frac{T_t^{C}}{t}\sqrt2 q I_1(2\sqrt2 q)
\le T_t^{C}\,t^{-1}\exp(2\sqrt2 q+o(q)).$$
The factor $\sqrt2 q (1+O(1/\sqrt2 q))/\sqrt{2\pi(2\sqrt2 q)}$ contributes only $\exp(o(q))$, which is why it is absorbed into the $o(q)$ term.
This proves the claim.
\end{proof}

\subsection{Asymptotic ratio to the conjectured value}

By Claims~\ref{claim:red-cycle-bound} and~\ref{claim:blue-cycle-bound}, the expected total number of monochromatic copies of $C_{2t}$ is at most
$$T_t^{C}\exp\left(-2q+O\left(\frac{q^2}{t}\right)\right)+T_t^{C}\,t^{-1}\exp(2\sqrt2 q+o(q)).$$
Choose
$$\gamma=\frac{1}{1+\sqrt2}~\text{and}~ q=\left\lfloor\frac{\gamma}{2}\log t\right\rfloor.$$
Then as $t\to\infty$, we have $q\to\infty$ and $q^2/t=o(1)$.  

We now estimate the two terms separately.
The red term is $$\exp\left(-2q+O\left(\frac{q^2}{t}\right)\right)=t^{-\gamma+o(1)}.$$
For the blue term, we deduce $$t^{-1}\exp(2\sqrt2 q+o(q))=t^{-1+\sqrt2\gamma+o(1)}.$$
Since $$-1+\sqrt2\gamma=-1+\frac{\sqrt2}{1+\sqrt2}=-\frac{1}{1+\sqrt2}=-\gamma,$$ the blue term is also $t^{-\gamma+o(1)}$.  
Combining the two terms, it follows that $$\E\bigl(N_{\red}(C_{2t})+N_{\blue}(C_{2t})\bigr) \le t^{-\gamma+o(1)}T_t^{C}.$$
By the probabilistic method, there exists a deterministic choice of the sets $M_b$ for which the actual number of monochromatic $C_{2t}$'s is at most this expectation. 
Therefore
$$m(C_{2t})\le t^{-\gamma+o(1)}T_t^{C} =t^{-\gamma+o(1)}\frac{(2t-1)!}{2}.$$
This proves Theorem~\ref{thm:main-cycle}.

Finally, the conjectured value in Conjecture~\ref{conj:CFSW-even} satisfies
$$\frac{2t-3}{2}(2t-2)!=\frac{2t-3}{2t-1}\cdot \frac{(2t-1)!}{2}=(1+o(1))\frac{(2t-1)!}{2}.$$
Hence our construction gives a polynomial improvement over the conjectured value.

\section{Proof of Theorem \ref{thm:main-odd-path}}\label{sec:count-odd-paths}

\subsection{The random-support coloring for odd paths}\label{sec:odd-path-coloring}

We now give an analogous local random construction for odd paths. Fix integers $t\ge 1$ and $q$ with $1\le q\le 2t$. 
Since $r(P_{2t+1})=3t$ by Theorem \ref{thm:Ramsey-paths}, the red-blue coloring is taken on the complete graph $K_{3t}$.
Let
$$V(K_{3t})=U\cup V,\,\, |U|=2t,\,~\text{and}~|V|=t$$
where $U$ and $V$ are pairwise disjoint. 
For each vertex $v\in V$, choose a $q$-element subset $M_v\subseteq U.$
We call $M_v$ the set of {\em blue supports} of $v$ in $U$.  
The edge-coloring of $K_{3t}$ is defined as follows:
\begin{itemize}[leftmargin=2em]
\setlength{\itemsep}{1pt}
\setlength{\parsep}{1pt}
\setlength{\parskip}{1pt}
\item all edges inside $U$ are blue;
\item all edges inside $V$ are red;
\item for $u\in U$ and $v\in V$, the edge $uv$ is blue if and only if $u\in M_v$; otherwise $uv$ is red.
\end{itemize}
The situation is presented in Figure~\ref{fig:odd-handstyle}.
Similarly, the construction will be used probabilistically. For each $v\in V$, choose $M_v$ independently and uniformly from all $q$-subsets of $U$.

\begin{figure}[htbp]
\centering
\begin{tikzpicture}[scale=1.05, every node/.style={font=\small}]
\tikzset{
  boxred/.style={draw=red, very thick, rounded corners=3pt, fill=SoftRed},
  boxblue/.style={draw=blue, very thick, rounded corners=3pt, fill=SoftBlue},
  boxgray/.style={draw=DarkGray, very thick, rounded corners=2pt, fill=white},
  rededge/.style={draw=red, very thick},
  blueedge/.style={draw=blue, very thick},
  info/.style={font=\small, text=DarkGray}
}

\draw[boxred] (1.4,4.35) rectangle (8.6,6.0);
\node at (5.0,5.7) {$V:~|V|=t$};

\draw[boxblue] (1.15,0.1) rectangle (8.85,2.45);
\node at (5.0,0.42) {$U:~|U|=2t$};

\draw[draw=DarkGray, very thick, fill=white] (3.3,1.275) ellipse (0.97 and 0.41);
\node at (3.3,1.275) {$M_v$};

\draw[blueedge] (4.8,4.95) -- (2.95,1.275);
\draw[blueedge] (4.8,4.95) -- (3.65,1.275);

\draw[rededge] (4.8,4.95) -- (6.8,1.275);
\draw[rededge] (4.8,4.95) -- (7.6,1.275);

\filldraw[draw=black, fill=black, very thick] (4.8,4.95) circle (0.06);
\node[above] at (4.8,4.95) {$v$};

\filldraw[draw=black, fill=black, very thick] (2.95,1.275) circle (0.05);
\filldraw[draw=black, fill=black, very thick] (3.65,1.275) circle (0.05);

\filldraw[draw=black, fill=black, very thick] (6.8,1.275) circle (0.06);
\filldraw[draw=black, fill=black, very thick] (7.6,1.275) circle (0.06);

\node[info, text=blue, align=center, fill=white, inner sep=1pt] at (2.3,3.45)
{$v$ is blue-adjacent to $M_v$};

\node[info, text=red, align=center, fill=white, inner sep=1pt] at (7.85,3.45)
{$v$ is red-adjacent to $U\setminus M_v$};

\end{tikzpicture}
\caption{The random-support construction for odd paths.}
\label{fig:odd-handstyle}
\end{figure}
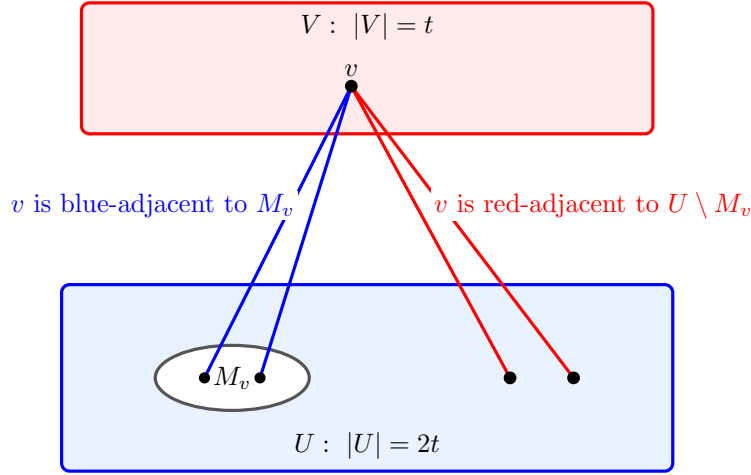

\subsection{Counting monochromatic odd paths}

Let
$$T_t^{P}:=\frac{t}{2}(2t)!.$$
This is the number of alternating $P_{2t+1}$'s in the complete bipartite graph $K_{2t,t}$.

\subsubsection{The expected number of red paths}

\begin{claim}\label{claim:red-path-exact}
The expected number of red copies of $P_{2t+1}$ is
$$\E\bigl(N_{\red}(P_{2t+1})\bigr)=T_t^{P}\sigma_q^t,$$
where
$$\sigma_q=\frac{\binom{2t-2}{q}}{\binom{2t}{q}}=\frac{(2t-q)(2t-q-1)}{2t(2t-1)}.$$
\end{claim}

\begin{proof}
In the red graph, the set $U$ is independent and $V$ is a clique. 
Since no two vertices of $U$ can be adjacent on a red path, every red $P_{2t+1}$ uses all vertices of $V$, uses exactly $t+1$ vertices of $U$, and must alternate in the form as follows:
$$U-V-U-V-\cdots-V-U.$$

Ignoring the support restrictions, the number of such alternating paths is
$$\frac{1}{2}\fall{2t}{t+1}\,t!=\frac{t}{2}(2t)!=T_t^{P}.$$
Now fix one candidate alternating path.  Every vertex $v\in V$ has two neighbours in $U$ on the path.  The two corresponding edges are red if and only if neither of these two $U$-vertices lies in $M_v$.  Since $M_v$ is a uniformly random $q$-subset of $U$, this probability is
$$\sigma_q=\frac{\binom{2t-2}{q}}{\binom{2t}{q}}=\frac{(2t-q)(2t-q-1)}{2t(2t-1)}.$$
The choices of $M_v$ are independent for different vertices $v\in V$, so the probability that the fixed candidate path is red is $\sigma_q^t$. 
Thus, the expected number of red copies of $P_{2t+1}$ is
$$\E\bigl(N_{\red}(P_{2t+1})\bigr)=T_t^{P}\sigma_q^t,$$
as expected. 
\end{proof}

\begin{claim}\label{claim:red-path-bound}
If $q=o(\sqrt t)$, then
$$\E\bigl(N_{\red}(P_{2t+1})\bigr)\le T_t^{P}\exp\left(-q+O\left(\frac{q^2}{t}\right)\right).$$
\end{claim}

\begin{proof}
By Claim~\ref{claim:red-path-exact}, we only need to estimate $\sigma_q^t$. 
The proof is essentially the same as the proof of Claim \ref{claim:red-cycle-bound} and we sketch the proof as follows.  
Note that
$$\sigma_q=\frac{(2t-q)(2t-q-1)}{2t(2t-1)}=1-\frac{q}{t}+O\left(\frac{q^2}{t^2}\right).$$
Since $q=o(\sqrt{t})$, we have $q/t=o(1)$. Applying the Taylor expansion of $\log(1+x)$ around $x=0$, we obtain
$$\log \sigma_q=-\frac{q}{t}+O\left(\frac{q^2}{t^2}\right)$$
and then
$$t\log\sigma_q=-q+O\left(\frac{q^2}{t}\right).$$
 Hence
$$\sigma_q^t=\exp\left(-q+O\left(\frac{q^2}{t}\right)\right).$$
This claim is proved.
\end{proof}

\subsubsection{The expected number of blue paths}

In the blue graph, $U$ is a blue clique and $V$ is blue-independent.  Since $|U|=2t$, every blue $P_{2t+1}$ must use at least one vertex of $V$.

Suppose a blue $P_{2t+1}$ uses exactly $s$ vertices of $V$, where $1\le s\le t$.  Then it uses $2t+1-s$ vertices of $U$.  We first order the chosen vertices of $U$ and then insert the $s$ vertices of $V$ into the gaps.  Let $j$ be the number of inserted vertices of $V$ which lie at the two endpoints of the path. Thus $j=0,1,2$ and $j\le s$.

For fixed $s$ and $j$, we have the following facts:
\begin{itemize}[leftmargin=2.2em]
\item the $s$ vertices of $V$ can be chosen and ordered in $\fall{t}{s}$ ways;
\item the $2t+1-s$ vertices of $U$ can be chosen and ordered in $\fall{2t}{2t+1-s}$ ways;
\item there are $\binom{2}{j}\binom{2t-s}{s-j}$ ways to choose the endpoint and internal gaps;
\item a vertex of $V$ placed at an endpoint requires one blue support edge, with probability
$$\alpha_q=\frac{\binom{q}{1}}{\binom{2t}{1}}=\frac{q}{2t};$$
\item a vertex of $V$ placed internally requires two blue support edges, with probability
$$\beta_q=\frac{\binom{q}{2}}{\binom{2t}{2}}=\frac{q(q-1)}{2t(2t-1)}.$$
\end{itemize}
Dividing by $2$ for reversing the path gives the following exact expectation.

\begin{claim}\label{claim:blue-path-exact}
The expected number of blue copies of $P_{2t+1}$ is
\begin{equation*}
\begin{aligned}
\E\bigl(N_{\blue}(P_{2t+1})\bigr)
={}&\frac{1}{2}\sum_{s=1}^{t}
\fall{t}{s}\fall{2t}{2t+1-s}
\sum_{j=0}^{\min\{2,s\}}
\binom{2}{j}\binom{2t-s}{s-j}\alpha_q^j\beta_q^{s-j}.
\end{aligned}
\end{equation*}

Equivalently,
\begin{equation*}
\frac{\E\bigl(N_{\blue}(P_{2t+1})\bigr)}{T_t^{P}}=
\sum_{s=1}^{t}\binom{t-1}{s-1}
\sum_{j=0}^{\min\{2,s\}}
\binom{2}{j}\binom{2t-s}{s-j}\alpha_q^j\beta_q^{s-j}.
\end{equation*}
\end{claim}

\begin{proof}
The first formula follows from the counting described above. For the normalized formula, divide by $T_t^{P}=t!\fall{2t}{t+1}/2$.
Since
$$\frac{\fall{t}{s}}{t!}=\frac{1}{(t-s)!}~~\text{and}~~\frac{\fall{2t}{2t+1-s}}{\fall{2t}{t+1}}=\frac{(t-1)!}{(s-1)!},$$
we obtain
$$\frac{\fall{t}{s}\fall{2t}{2t+1-s}}{t!\fall{2t}{t+1}}=\frac{(t-1)!}{(s-1)!(t-s)!}=\binom{t-1}{s-1}.$$
This yields the stated normalized expression.
\end{proof}

We now estimate the normalized blue expectation. 
If $q=O(\log t)$, then
$$\beta_q=\frac{q(q-1)}{2t(2t-1)}\le (1+o(1))\frac{q^2}{4t^2}.$$
We split the sum according to the value of $j$.

\begin{claim}\label{claim:blue-path-bound}
If $q=O(\log t)$, then
$$\E\bigl(N_{\blue}(P_{2t+1})\bigr)\le T_t^{P}\,t^{-1}\exp(\sqrt2 q+o(q)).$$
\end{claim}

\begin{proof}
Let $S_j$ be the contribution of the terms with exactly $j$ endpoint vertices from $V$ in the normalized expression of Claim~\ref{claim:blue-path-exact}.
We first prove the following fact.

\begin{fact}\label{Fac:fact}
For $x>0$,  we have
$$\sum_{m=0}^{\infty} \frac{x^m}{(m!)^2}\leq \exp(2\sqrt{x}).$$
\end{fact}
\begin{proof}
By the binomial theorem, we have
$$\frac{(2m)!}{(m!)^2}=\binom{2m}{m}\leq \sum_{i=0}^{2m} \binom{2m}{i}=(1+1)^{2m}=4^m.$$
Then
$$\sum_{m=0}^{\infty} \frac{x^m}{(m!)^2}
\leq \sum_{m=0}^{\infty} \frac{(2 \sqrt{x})^{2m}}{(2m)!}
\leq \sum_{n=0}^{\infty} \frac{(2 \sqrt{x})^{n}}{n!}
=\exp(2\sqrt{x}).$$
The first inequality is based on $1/(m!)^2\leq 4^m/(2m)!$.
The last equality follows from the Taylor expansion of $\exp(2\sqrt{x})$.
\end{proof}

We count the three types of blue odd paths described below.

\medskip\noindent\textbf{Type I: $j=0$.}
Then
$$S_0=\sum_{s=1}^{t}\binom{t-1}{s-1}\binom{2t-s}{s}\beta_q^s.$$
Using
$$\binom{t-1}{s-1}\le \frac{t^{s-1}}{(s-1)!},~~\binom{2t-s}{s}\le \frac{(2t)^s}{s!}, ~~\text{and}~~\beta_q\le (1+o(1))\frac{q^2}{4t^2},$$
we get
$$S_0\le \frac{1+o(1)}{t}\sum_{s=1}^{\infty}\frac{(q^2/2)^s}{(s-1)!s!}.$$
Let $m=s-1$.
Applying the Fact, we have
\begin{align*}
S_0\le &\frac{(1+o(1))q^2}{2t}
\sum_{m=0}^{\infty}\frac{(q^2/2)^m}{m!(m+1)!}
\leq \frac{(1+o(1))q^2}{2t} \sum_{m=0}^{\infty}\frac{(q^2/2)^m}{(m!)^2}\\
\leq &\frac{(1+o(1))q^2}{2t} \exp(\sqrt{2} q)
\leq t^{-1}\exp(\sqrt2 q+o(q)).
\end{align*}
In the final inequality, the multiplicative factor $(1+o(1))q^2/2$ contributes only $\exp(o(q))$.

\medskip\noindent\textbf{Type II: $j=1$.}
Clearly,
$$S_1=\sum_{s=1}^{t}\binom{t-1}{s-1}\binom{2}{1}\binom{2t-s}{s-1}\alpha_q\beta_q^{s-1}.$$
Recall that $\alpha_q=q/2t$. Similarly, easy calculations show that
$$S_1\le \frac{(1+o(1))q}{t}\sum_{s=1}^{\infty}\frac{(q^2/2)^{s-1}}{((s-1)!)^2}\le \frac{(1+o(1))q}{t}\sum_{m=0}^{\infty}\frac{(q^2/2)^{m}}{(m!)^2} \le t^{-1}\exp(\sqrt{2} q+o(q)).$$

\medskip\noindent\textbf{Type III: $j=2$.}
Obviously, $$S_2=\sum_{s=2}^{t}\binom{t-1}{s-1}\binom{2t-s}{s-2}\alpha_q^2\beta_q^{s-2}.$$
The same estimates yield
$$S_2\le \frac{(1+o(1))q^2}{t}\sum_{s=2}^{\infty}\frac{(q^2/2)^{s-2}}{(s-1)!(s-2)!}\le t^{-1}\exp(\sqrt2 q+o(q)).$$

Combining the estimates for the three types, we obtain
$$\frac{\E\bigl(N_{\blue}(P_{2t+1})\bigr)}{T_t^{P}}=S_0+S_1+S_2\le t^{-1}\exp(\sqrt2 q+o(q)).$$
This proves claim \ref{claim:blue-path-bound}.
\end{proof}

\subsection{Asymptotic ratio to the conjectured value}
The proof follows the same reasoning as in Theorem~\ref{thm:main-cycle}.
By Claims~\ref{claim:red-path-bound} and~\ref{claim:blue-path-bound}, the expected total number of monochromatic copies of $P_{2t+1}$ is at most
$$T_t^{P}\exp\left(-q+O\left(\frac{q^2}{t}\right)\right)+T_t^{P}t^{-1}\exp(\sqrt2 q+o(q)).$$
We choose the parameters as below:
$$\gamma=\frac{1}{1+\sqrt2} ~~\text{and}~~q=\left\lfloor \gamma\log t\right\rfloor.$$
Then $q\to\infty$ and $q^2/t=o(1)$. 
The red term and blue term are
$$\exp\left(-q+O\left(\frac{q^2}{t}\right)\right)=t^{-\gamma+o(1)}$$
and
$$t^{-1}\exp(\sqrt2 q+o(q))=t^{-1+\sqrt2\gamma+o(1)},$$
respectively.
Since $-1+\sqrt2\gamma=-\gamma$, we deduce
$$\E\bigl(N_{\red}(P_{2t+1})+N_{\blue}(P_{2t+1})\bigr)\le t^{-\gamma+o(1)}T_t^{P}.$$
It implies that
$$m(P_{2t+1})\le t^{-\gamma+o(1)}T_t^{P}=t^{-\gamma+o(1)}\frac{t}{2}(2t)!.$$
The proof of Theorem~\ref{thm:main-odd-path} complete.

\section{Proof of Theorem \ref{thm:main-theorem-even-path}}\label{sec:even paths}

\subsection{A red/blue edge-coloring construction for even paths}\label{subsec:counterexample-graph-even-path}
In this subsection, we first introduce a red/blue edge-coloring of the complete graph $K_n$.
Let
$$V(K_n)=S\cup R\cup B\cup\{x\} \text{ and } A=S\cup R,$$
where $S,R,B$ and $\{x\}$ are pairwise disjoint. 
The edge-coloring of $K_n$ is defined as follows:
\begin{itemize}[leftmargin=2em]
\setlength{\itemsep}{1pt}
\setlength{\parsep}{1pt}
\setlength{\parskip}{1pt}
\item all edges inside $A$ are blue;
\item all edges inside $B$ are blue;
\item all edges between $A$ and $B$ are red;
\item all edges between $x$ and $B$ are red;
\item all edges between $x$ and $S$ are blue;
\item all edges between $x$ and $R$ are red.
\end{itemize}
Equivalently, the red graph consists of the complete bipartite graph between $A\cup\{x\}$ and $B$, together with the red star from $x$ to $R$; the blue graph consists of the complete graphs $K_{|A|}$ on $A$ and $K_{|B|}$ on $B$, together with the blue star from $x$ to $S$. The construction is shown in Figure \ref{fig:block}.

\begin{figure}[ht]
\centering
\begin{tikzpicture}[scale=1.0, every node/.style={font=\small}]
  \node[circle,draw,fill=white,inner sep=2pt] (x) at (0.4,2.2) {$x$};

  \node[draw=blue,very thick,rounded corners,minimum width=2.0cm,minimum height=1.0cm,fill=blue!6] (S) at (-2.0,0) {$S$};
  \node[draw=blue,very thick,rounded corners,minimum width=2.0cm,minimum height=1.0cm,fill=blue!6] (R) at (0.4,0) {$R$};
  \node[draw=blue,very thick,rounded corners,minimum width=2.1cm,minimum height=1.0cm,fill=blue!6] (B) at (4.0,0) {$B$};

\node[draw=blue,very thick,rounded corners,fit=(S)(R),inner xsep=0.2cm,inner ysep=0.4cm] (Abox) {};
\node at (Abox.south) [yshift=0.25cm] {$A$};

  \draw[draw=blue,very thick] (x) -- (S.north);
  \draw[draw=red,very thick] (x) -- (R.north);
  \draw[draw=red,very thick] (x) -- (B.north);

  \draw[draw=red,line width=1.2mm] (Abox.east) -- (B.west);
\end{tikzpicture}
\caption{A red/blue edge-coloring of the complete graph $K_n$. }
\label{fig:block}
\end{figure}
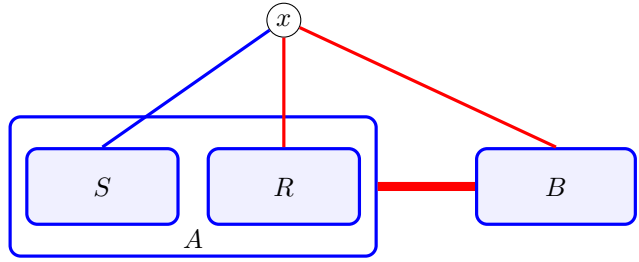

\subsection{Counting monochromatic even paths}

By Theorem \ref{thm:Ramsey-paths}, $r(P_{2t})=2t-1+t=3t-1$.
Given an integer $\alpha$ (the exact value of $\alpha$ will be specified later),
consider the edge-coloring of $K_{3t-1}$ given in Subsection \ref{subsec:counterexample-graph-even-path} with
$$|S|=2t-1-\alpha,\,\,|R|=\alpha, \text{ and } |B|=t-1.$$
Clearly,
$$|A|=|S|+|R|=2t-1 \text{ and } |A\cup\{x\}|=2t.$$

\begin{claim}\label{C1}
The number of blue copies of $P_{2t}$ is
$$\left(2t-1-\alpha+\binom{2t-1-\alpha}{2}\right)(2t-2)!.$$
\end{claim}

\begin{proof}
The blue graph has two blue components relevant here: $B$, of order $t-1$, and $A\cup\{x\}$, of order $2t$. Since $t-1<2t$, no blue copies of $P_{2t}$ can lie in $B$. 
Also there are no blue edges between $B$ and $A\cup\{x\}$. 
Therefore every blue copy of $P_{2t}$ must use exactly all vertices of $A\cup\{x\}$.

Inside $A\cup\{x\}$, the set $A$ is a clique of order $2t-1$, and $x$ is adjacent in blue exactly to the $2t-1-\alpha$ vertices of $S$. 
We count Hamilton paths of this graph (induced by blue edges).

We first suppose that $x$ is an endpoint of the Hamilton path. 
Choose its unique neighbor in the path: there are $2t-1-\alpha$ choices. 
After that neighbor is chosen, the remaining $2t-2$ vertices of $A$ can be ordered arbitrarily. 
Hence, the number of such paths is
$$(2t-1-\alpha)(2t-2)!.$$

Now, assume that $x$ is an internal vertex of the Hamilton path. 
The two neighbors $s_1$ and $s_2$ of $x$ must be two distinct vertices of $S$. 
Choose this unordered pair in $\binom{2t-1-\alpha}{2}$ ways. 
Contract the three-vertex segment $s_1-x-s_2$ to a single block. For a fixed unordered pair $\{s_1,s_2\}$, the block has two orientations, and the block together with the other $2t-3$ vertices of $A$ may be ordered in $(2t-2)!$ ways. Since reversing the whole path gives the same unlabelled path, the factor $2$ from the block orientation and the factor $2$ from reversal cancel. Hence there are $(2t-2)!$ paths for each unordered pair. 
The internal case contributes $$\binom{2t-1-\alpha}{2}(2t-2)!.$$
Adding the endpoint and internal cases gives
$$\left(2t-1-\alpha+\binom{2t-1-\alpha}{2}\right)(2t-2)!$$
blue copies of $P_{2t}$. This proves the claim.
\end{proof}

\begin{claim}\label{C2}
The number of red copies of $P_{2t}$ is
$$\left(\alpha t+\binom{\alpha}{2}\right)(2t-2)!.$$
\end{claim}

\begin{proof}
Let $L=A\cup\{x\}$. 
The red graph contains all edges between $L$ and $B$, and it also contains the red edges from $x$ to $R$. 
There are no other red edges inside $L$, and there are no red edges inside $B$.

Since $|B|=t-1$, a red path in the bipartite graph between $L$ and $B$ has at most
$2|B|+1=2t-1$
vertices. 
Therefore a red copy of $P_{2t}$ must use at least one red edge of the form $xr$ with $r\in R$. 
Since all such edges meet $x$, a path can use either one or two of them. We count the two types of red paths described below.

\medskip\noindent\textbf{Type I: exactly one edge $xr$ is used.}
Choose $r\in R$ in $\alpha$ ways. 
The path must use all $t-1$ vertices of $B$, and it must use $t+1$ vertices of $L$. 
Besides $x$ and $r$, the number of ways to choose $t-1$ further vertices from $A\setminus\{r\}$ is $$\binom{2t-2}{t-1}.$$


Treat the adjacent pair $xr$ as one $L$-block. Then we have $t$ $L$-blocks and $t-1$ vertices of $B$. A red path using them must alternate, and hence its two endpoints are $L$-blocks. For fixed chosen vertices, order the $t$ $L$-blocks in $t!$ ways and insert the $t-1$ vertices of $B$ in the $t-1$ gaps in $(t-1)!$ ways. The block $xr$ has two orientations, but reversal of the whole path identifies two orientations of the same unlabelled path, so these factors cancel. Thus the number for fixed vertices is $$t!(t-1)!.$$ 

The Type I contribution is
$$\alpha\binom{2t-2}{t-1}t!(t-1)!=\alpha t(2t-2)!.$$

\medskip\noindent\textbf{Type II: two edges $xr_1$ and $xr_2$ are used.}
Choose the unordered pair $\{r_1,r_2\}\subset R$ in $$\binom{\alpha}{2}$$ ways. 
The path must use all vertices of $B$, and it must use $t+1$ vertices of $L$. Besides $x, r_1, r_2$, the number of ways to choose $t-2$ further vertices from $A\setminus\{r_1,r_2\}$ is $$\binom{2t-3}{t-2}.$$

Now, we count paths for fixed vertices.
Treat the three-vertex segment $\{r_1,x,r_2\}$ as one $L$-block. 
Then there are $t-1$ $L$-blocks and $t-1$ vertices of $B$. 
An alternating path with equal part sizes has one endpoint in each part. 
There are two alternative approaches: 
$$L-B-L-B-\cdots-L-B$$
and
$$B-L-B-L-\cdots-B-L.$$
For either case, the $L$-blocks may be ordered in $(t-1)!$ ways and the $B$ vertices in $(t-1)!$ ways. 
The three-vertex block paths is $r_1-x-r_2$ and $r_2-x-r_1$.
Thus the number of oriented path is 
$2((t-1)!)^2 \cdot 2=4((t-1)!)^2$.
Dividing by $2$ for reversal, the number of unoriented paths for fixed vertices is 
$$2((t-1)!)^2.$$

Therefore, the contribution of Type II is
$$\binom{\alpha}{2}\binom{2t-3}{t-2}2((t-1)!)^2
 =\binom{\alpha}{2}(2t-2)!.$$
Summing the two types of paths proves the claim.
\end{proof}


Combining Claims \ref{C1} and \ref{C2}, we conclude that the number of monochromatic copies 
of $P_{2t}$ in this red/blue coloring of the complete graph is 
$$\left(2t-1-\alpha+\binom{2t-1-\alpha}{2}+\alpha t+\binom{\alpha}{2}\right)(2t-2)!=(\alpha^2-t \alpha+2t^2-t)(2t-2)!.$$
It is minimized by choosing $\alpha=\lfloor t/2\rfloor$,
which therefore yields the best possible bound $$m(P_{2t}) \le\frac{(2t)!}{2}-\left\lfloor \frac{t^2}{4}\right\rfloor(2t-2)!.$$

\subsection{Asymptotic ratio to the conjectured value}
Since $$m(P_{2t}) \leq\frac{(2t)!}{2}-\left\lfloor \frac{t^2}{4}\right\rfloor(2t-2)!$$
and $$\frac{\left\lfloor t^2/4\right\rfloor}{t(2t-1)}=\frac{1}{8}+o(1),$$
we obtain
$$m(P_{2t}) \leq \left(\frac{7}{8}+o(1) \right)\frac{(2t)!}{2}.$$
The proof is complete.

\section{Concluding Remarks}\label{sec:remarks}

The two constructions about $m(C_{2t})$ and $m(P_{2t+1})$ are both local random constructions, but the local random sets play slightly different roles.

For even cycles, the set $A$ is a red core of size $t$, while $B$ is a blue reservoir of size $2t-1$. Since $B$ is one vertex short of containing $C_{2t}$ by itself, every blue $C_{2t}$ must enter $A$ through portal edges. On the other hand, every red $C_{2t}$ is forced to use all of $A$ and alternate with $t$ vertices of $B$. The random portal sets therefore reduce the red cycles multiplicatively, while the sparsity of the portals controls the blue cycles.

For odd paths, the set $U$ is a blue clique of size $2t$, while $V$ is a red clique of size $t$. Every red $P_{2t+1}$ is forced to alternate through all vertices of $V$, so the random supports remove a multiplicative proportion of red paths. A blue $P_{2t+1}$ must insert at least one vertex of $V$ into the blue clique $U$, and each insertion requires one or two support edges.

The exponents in Theorems~\ref{thm:main-cycle} and~\ref{thm:main-odd-path} arise from balancing the two contributions. In the even-cycle case, the red contribution is essentially $\exp(-2q)$ and the blue contribution is essentially $t^{-1}\exp(2\sqrt2q)$. In the odd-path case, the red contribution is essentially $\exp(-q)$ and the blue contribution is essentially $t^{-1}\exp(\sqrt2q)$. Both balances lead to the same final exponent
$$\gamma=\frac{1}{1+\sqrt2}.$$

The constructions only give upper bounds. We cannot prove that these bounds are best possible. It would be interesting to determine the true order of $m(C_{2t})$ and $m(P_{2t+1})$, or even to decide whether these local random models are optimal within suitable restricted classes of colorings.

We once tried to apply the local random construction to find a red/blue edge-coloring of $K_{3t-1}$ for $P_{2t}$, the expected number of monochromatic copies of $P_{2t}$ turned out to be even larger than the conjectured value. We do not know if the upper bound obtained in Theorem \ref{thm:main-theorem-even-path} is near the value of $m(P_{2t})$. So, it is of interesting to determine the true order of $m(P_{2t})$ and corresponding optimal red/blue edge-colorings.

\section*{Acknowledgement}
\noindent This research is supported by National Key R\&D Program of China under grant number 2024YFA1013900 and NSFC under grant number 12471327.
\section*{Declaration}
	
\noindent$\textbf{Conflict~of~interest}$
The authors declare that they have no known competing financial interests or personal relationships that could have appeared to influence the work reported in this paper.
\vskip 2mm	
\noindent$\textbf{Data~availability}$
No data was used for the research described in the article.

\end{document}